\documentclass[a4paper]{article}

\usepackage{amssymb,amsfonts,amsmath,amsthm,cite,mathrsfs}
\usepackage{graphicx,epsfig,subfigure,tikz}
\usepackage{enumerate}
\usepackage{enumitem}
\usepackage{dsfont}
\usepackage{color,xcolor}
\usepackage{changes}
\newtheorem{thm}{Theorem}[section]

\newtheorem{lem}[thm]{Lemma}
\newtheorem{claim}[thm]{Claim}
\newtheorem{prop}[thm]{Proposition}

\newtheorem{ex}[thm]{Example}

\newtheorem{fact}[thm]{Fact}

\newtheorem*{claim*}{Claim}

\makeatletter \@addtoreset{equation}{section}

\def\qed{\hfill \rule{4pt}{7pt}}
\def\hf{\mathcal{F}}

\def\hm{\mathcal{M}}
\def\hg{\mathcal{G}}
\def\hht{\mathcal{T}}
\def\hh{\mathcal{H}}
\def\hhn{\mathcal{N}}

\def\ha{\mathcal{A}}
\def\hb{\mathcal{B}}

\def\hn{\mathbb{N}}
\def\ex{\mathcal{\mathbb{E}}}
\def\hr{\mathcal{R}}

\def\hq{\mathcal{Q}}
\def\hht{\mathcal{T}}

\begin{document}

\title{Minimum $\ell$-degree thresholds for rainbow perfect matching in $k$-uniform hypergraphs}
\author{ Jie You\\[5pt]
Center for Applied Mathematics\\
Tianjin University\\
Tianjin 300072, P. R. China\\[6pt]
E-mail: yj\underline{ }math@tju.edu.cn
}
\date{}
\maketitle
\begin{abstract}
Given $n\in k\hn$ elements set $V$ and  $k$-uniform hypergraphs $\hh_1,\ldots,\hh_{n/k}$ on $V$. A rainbow perfect matching is a collection of pairwise disjoint edges $E_1\in \hh_1,\ldots,E_{n/k}\in \hh_{n/k}$ such that $E_1\cup\cdots\cup E_{n/k}=V$.  In this paper, we determine the minimum $\ell$-degree condition that guarantees the existence of a rainbow perfect matching for sufficiently large $n$ and $\ell\geq k/2$.
\end{abstract}

\section{Introduction}

\subsection{Background}
Finding a spanning subgraph in a given graph or hypergraph is a fundamental problem in graph theory. In particular, it is desirable to fully characterize all graphs or hypergraphs that contain a spanning copy of a specific subgraph. For instance, Tutte's theorem \cite{tutte1947factorization} provides a characterization of all graphs that contain a perfect matching. However, for many hypergraphs, such a characterization is unlikely to exist due to the NP-completeness of the decision problem of whether a hypergraph $\hh$ contains a given subgraph $\hf$. In fact, Garey and Johnson \cite{garey1979computers} showed that the decision problem of whether a $k$-uniform hypergraph contains a perfect matching  is NP-complete for $k\geq 3$.

Given a set $V$ of size $n$ and an integer $k\geq 2$, we use $\binom{V}{k}$ to denote the family of all $k$-element subsets ($k$-subsets, for short) of $V$. A subfamily $\hh\subseteq\binom{V}{k}$ is called a {\it $k$-uniform hypergraph} (or {\it $k$-graph} in short). For $\hh\subseteq\binom{V}{k}$, we often use $V(\hh)$ to denote its vertex set $V$ and use $\hh$ to denote its edge set. Define the complement of $\hh$ as $\overline{\hh}:=\binom{V}{k}\setminus \hh$. Given $A\subseteq V$, let $\hh[A]$ denote the sub $k$-graph of $\hh$ induced by $A$, namely, $\hh[A]:=\hh\cap \binom{A}{k}$. Define $\hh- A:=\hh[V(H)\setminus A]$.

For $S\in \binom{V}{\ell}$ with $0\leq \ell\leq k-1$, define the {\it link graph} of $S$ $\hhn_\hh(S):=\{T\colon S\cup T\in \hh\}$, and let $\deg_\hh(S)$ be the cardinality of $\hhn_\hh(S)$. The \emph{minimum $\ell$-degree} $\delta_\ell(\hh)$ of $\hh$ is the minimum of $\deg_\hh(S)$ over all $\ell$-element subsets $S$ of $V(\hh)$. Clearly, $\delta_0(\hh)$ is the number of edges in $\hh$. We refer to $\delta_1(\hh)$ as the \emph{minimum vertex degree} of $\hh$ and $\delta_{k-1}(\hh)$ as the \emph{minimum codegree} of $\hh$. We often omit the subscript $\hh$ when the context is clear.

 Given  a partition $V=A\cup B$, let $\hh^0(A,B)$ (or $\hh^1(A,B)$) denote the $k$-graph on $V$ with edge set consists of all  the edges intersect $A$ in an even (or odd) number of vertices.  Clearly,  $\hh^i(A,B)$ contains  perfect matching if $n\in k\hn$ and $in/k$  has some parity with $|A|$.

Let $n\in k\hn$. Define 
\[
ext(n,k):=\left\{ \hh^i(A,B)\colon A\cup B=V, in/k \mbox{ has different parity with } |A| \right\}.
\]
Let
\[\delta(n,k,\ell)=\max_{\hh\in ext(n,k)}\delta_\ell(\hh).\]
Treglown and Zhao determine the minimum $\ell$-degree condition that guarantees the existence of a  perfect matching in $k$-graph.
\begin{thm}[\hspace{1sp}\cite{Treglown2012, Treglown2013}]\label{ZT}
Given integers $k,\ell$ such that $k\geq 3$ and $k/2\leq \ell\leq k-1$, there exists  $n_0\in \hn$ such that the following holds. Suppose that $\hh$ is a $k$-graph on $n$ vertices with  $n\in k\hn$ and $n\geq n_0$ satisfying $\delta_\ell(\hh)>\delta(n,k,\ell)$, then $\hh$ contains a perfect matching.
\end{thm}

Given  $k$-uniform hypergraphs $\hh_1,\ldots,\hh_{n/k}$ on $V$, a rainbow perfect matching is a collection of pairwise disjoint  edges $E_1\in \hh_1,\ldots,E_{n/k}\in \hh_{n/k}$ such that $E_1\cup\cdots\cup E_{n/k}=V$. Lu, Wang and Yu give the co-degree threshold for rainbow perfect matchings in $k$-graphs.
\begin{thm}[\hspace{1sp}\cite{Lu2021}]\label{Lu2021}
  Given integers $k,\ell$ such that $k\geq 3$ and $k/2\leq \ell\leq k-1$ and $n\in k\hn$, there exists  $n_0\in \hn$ such that the following holds. Suppose that $\hh_1,\ldots,\hh_{n/k}$ are  $k$-graphs on $n$ vertices with $n\geq n_0$  and $n\in k\hn$ satisfying $\delta_{k-1}(\hh_i)>\delta(n,k,k-1)$ for $i=1,\ldots,n/k$. Then $\hh_1,\ldots,\hh_{n/k}$ admit a rainbow  matching.
\end{thm}
 In this paper,  we determine the minimum $\ell$-degree condition that guarantees the existence of a rainbow perfect matching in $\hh_1,\ldots,\hh_{n/k}$ for sufficiently large $n\in k\hn$ and $\ell\geq k/2$, which is a generalisation of Theorem \ref{ZT} and \ref{Lu2021}.
\begin{thm}[Main Result]\label{thm-main}
  Given integers $k,\ell$ such that $k\geq 3$ and $k/2\leq \ell\leq k-1$ and $n\in k\hn$, there exists  $n_0\in \hn$ such that the following holds. Suppose that $\hh_1,\ldots,\hh_{n/k}$ are  $k$-graphs on $n$  vertices with $n\geq n_0$ and $n\in k\hn$ satisfying $\delta_\ell(\hh_i)>\delta(n,k,\ell)$ for $i=1,\ldots,n/k$. Then $\{\hh_1,\ldots,\hh_{n/k}\}$ admits a rainbow  matching.
\end{thm}
It is shown in \cite{Han2022} that $\delta(n,k,\ell)$ takes its maximum value at $\hh^i(A,B)$ for some $A,B,i$ satisfy $-1\leq|A|-|B|\leq 1$ and $2\nmid in/k+|A|$, and denote one of the extremal graph as $\hh_{ext}$. By letting $\hh_1=\ldots=\hh_{n/k}=\hh_{ext}$. We infer that there is  no rainbow perfect matching in $\hh_1,\ldots,\hh_{n/k}$. It implies that the minimum $\ell$-degree condition in Theorem \ref{thm-main} is best possible.

It seems hard to compute the precise values of $\delta(n,k,\ell)$ for $\ell\leq k-2$. In \cite{Treglown2012}, it is showed that
\begin{equation}
	\delta(n,k,k-1)=
	\begin{cases}
		n/2-k+2 & \mbox{if $k/2$ is even and $n/k$ is odd}\\
		n/2-k+3/2 & \mbox{if $k$ is odd and $(n-1)/2$ is odd }\\
		n/2-k+1/2 & \mbox{if $k$ is odd and $(n-1)/2$ is even}\\
        n/2-k+1 & \mbox{otherwise.}
	\end{cases}
\end{equation}

The following proposition allows us to infer $\ell'$ minimum degree from $\ell$ minimum degree for $\ell'\leq \ell$. The proof is straightforward and omitted here. This result has also  used in \cite{Treglown2012, Treglown2013}.
\begin{prop}\label{prop-2.1}
Let $0\leq \ell'\leq \ell<k$ and let $\hh$ be a $k$-graph. If $\delta_\ell(\hh)\geq x\binom{n-\ell}{k-\ell}$ for some $0\leq x\leq 1$, then $\delta_{\ell'}(\hh)\geq x\binom{n-\ell'}{k-\ell'}$.
\end{prop}
By Proposition \ref{prop-2.1}, $\delta(n,k,\ell)$  are  only known to be $(1/2-o(1))\binom{n-\ell}{k-\ell}$.


 We often write $0<a_1\ll a_2$ to mean that  there are increasing functions $f$  such that $a_1\leq f(a_2)$. Throughout the paper, we omit floors and ceilings whenever this does not affect the argument.

\subsection{Proof overview}

We say a hypergraph is a $(1,r)$-graph if its vertex set  can be partitioned into two parts $X\cup V$, such that $r|X|=|V|$, and every edge intersects $X$ in exactly one vertex and intersects $V$ in $r$ vertices. Moreover, we say a subset $U\subseteq X\cup V$ is balanced if $r|U\cap X|=|U\cap V|$. 

In \cite{Lu2021-1}, Lu, Yu, and Yuan introduced a $(1,k)$-graphs and reduced finding rainbow matchings in $\hh_1, ..., \hh_{n/k}$ to finding  matchings in this $(1,k)$-graph. Define $\mathcal{T}(\hh_1,\ldots,\hh_{n/k})$ be the hypergraph with vertex set $X\cup V$ where $X=\{x_1,\ldots,x_{n/k}\}$ and edges set $\cup_{i=1}^{t}\{\{x_i\}\cup E\colon E\in \hh_i\}$.  It is clear that $\hh_1,\ldots,\hh_{n/k}$ contain a rainbow  matching if and only if $\mathcal{T}(\hh_1,\ldots,\hh_{n/k})$ contains a perfect matching.

Let $\epsilon>0$ and suppose that $\hh$ and $\hh'$ are $r$-graphs on $V(\hh)$. We say that $\hh$ is $\epsilon$-close to $\hh'$
if $\hh$ can be made a copy of $\hh'$ by adding and deleting at most $\epsilon |V(\hh)|^r$ edges.

Denote $\mathcal{T}(\hh_{ext},\ldots,\hh_{ext})$ by $\hht_{ext}$, where there are $n/k$   $H_{ext}$'s. As a common approach to obtain exact results, Theorem \ref{thm-main} is proven by distinguishing an \emph{extremal} case from a \emph{non-extremal} case and solve them separately.

\begin{thm}[Non-extremal Case]\label{thm-nonextr}
Given $k,\ell$ such that $k\geq 3$ and $k/2\leq \ell\leq k-1$ and $\epsilon>0$, there exists $n_0$ such that the following holds. Suppose that $\hh_1,\ldots,\hh_{n/k}$ are $k$-graphs on vertex set $V$ with $|V|=n\geq n_0$ and $n\in k\hn$. Let $\hht=\mathcal{T}(\hh_1,\ldots,\hh_{n/k})$. If $\hht$ is not $\epsilon$-close to $\hht_{ext}$ and $\deg_\ell(\hh_i)> \delta(n,k,\ell)$ for every $1\leq i\leq n/k $, then $\hht$ contains a perfect matching.
\end{thm}
\begin{thm}[Extremal Case]\label{thm-extr}
Given $k,\ell$ such that $k\geq 3$ and $k/2\leq \ell\leq k-1$.
There exists $\epsilon>0$ and $n_0$ such that for $n\geq n_0$ the following holds. Suppose that $\hh_1,\ldots,\hh_{n/k}$ are $k$-graphs on vertex set $V$ with $|V|=n\geq n_0$ and $n\in k\hn$. Let $\hht=\mathcal{T}(\hh_1,\ldots,\hh_{n/k})$. If $\hht$ is $\epsilon$-close to $\hht_{ext}$ and $\deg_\ell(\hh_i)> \delta(n,k,\ell)$ for every $1\leq i\leq n/k $, then $\hht$ contains a perfect matching.
\end{thm}

In the packing problem, the absorb method is a commonly used technique, which was first introduced by R\"{o}dl, Ruci\'{n}ski and Szemer\'{e}di \cite{VOJTECH2006}, and subsequently employed by many researchers (see \cite[etc]{Lu2021,Wang2023,Marks2011,Treglown2013,Treglown2012,Reiher2019,Kuehn2010,Han2022}). To establish this method, we need to prove two lemmas: the Absorbing Lemma and the Almost Cover Lemma. The Absorbing Lemma involves finding a matching $\ha$ in $\mathcal{T}$ of an appropriate  size, such that for any sufficiently small subset $U\subset X\cup V$, where $U$ is a balanced set, there exists a matching $\mathcal{Q}$ with $V(\mathcal{Q}) = V(\ha)\cup U$. The matching $\ha$ is referred to as an absorb set. The Almost Cover Lemma involves finding a matching in $\mathcal{T}$ that covers almost all the vertices. The proof of the Almost Cover Lemma follows the process outlined in \cite{Marks2011}. By excluding an absorb set in advance, applying the Almost Cover Lemma to the remaining graph, and then absorbing the uncovered vertices into the absorb set, we can prove Theorem  \ref{thm-main} .

In order to establish the Absorbing Lemma, we require  a counting lemma.  Given a balanced $(k+1)$-set, this lemma counts the number of size-2 matchings that can absorb this set.    By applying the Frankl-Kupavskii concentration inequality, we can then prove the Absorbing Lemma. It should be noted that the counting lemma valid under the extra condition that $\mathcal{T}$ is not close to the extremal hypergraph $\mathcal{T}_{ext}$. This is why we separate the proof into extremal and non-extremal cases.

For the extremal case, we mainly  use a result of the author joint with Wang in \cite{Wang2023}. This result states that if every vertex in an $r$-partite $r$-graph is ``good'', then a perfect matching exists. Here, the term ``good'' refers to a binary relation between graphs that is similar to, but stronger than, the concept of ``close'' as previously defined. The precise definition of ``good'' will be provided later, following the earlier explanation of ``close''.  Therefore, we have two remaining tasks: The first one is to remove the vertices that are not ``good'' using a matching. Then we get a hypergraph with all the vertices are ``good''.  The second one is to   divide the remaining hypergraph into two $(k+1)$-partite $(k+1)$-graphs. By showing that vertices in these two $(k+1)$-partite $(k+1)$-graphs inherit the ``good'' property, we can apply the result in \cite{Wang2023} to complete the proof for the extremal case.

The proofs of Theorems \ref{thm-nonextr} and \ref{thm-extr} are given in Sections \ref{sec:non-extrcase} and \ref{sec:extr case}, respectively.

\section{Non-extremal Case}\label{sec:non-extrcase}

In this section, we deal with the non-extremal case by following the absorbing method initiated by R\"{o}dl, Ruci\'{n}ski and Szemer\'{e}di \cite{VOJTECH2006}.

These following two lemmas will be used in the Non-extremal Case. The proofs of these two lemmas will be respectively deferred to Sections 2.1 and 2.2.
\begin{lem}[Almost cover]\label{lem-pathcover}
Given $\xi>0$ and $k,\ell$ such that $k\geq 3$ and $k/2\leq \ell\leq k-1$, there exists $n_0$ such that the following holds. Suppose that $\hh_1,\ldots,\hh_{n/k}$ are $k$-graphs on vertex set $V$ with $|V|=n\geq n_0$ and $n\in k\hn$. Let $\hht=\mathcal{T}(\hh_1,\ldots,\hh_{n/k})$. If  for every  $1\leq i\leq n/k $,
\[
\delta_\ell(\hh_i)\geq \left(\frac{k-\ell}{k}-\frac{1}{k^{k-\ell}}+\xi\right)\binom{n-\ell}{k-\ell},
\]
then $\hht$ contains a  matching  cover all but at most $\sqrt{n}$ vertices.
\end{lem}

\begin{lem}[Absorbing ]\label{lem-absorbing}
	 Given  $\epsilon>0$ and $k,\ell$ such that $k\geq 3$ and $k/2\leq \ell\leq k-1$,  there exists $\gamma$ and $n_0$ such that the following holds. Suppose that $\hh_1,\ldots,\hh_{n/k}$ are $k$-graphs on vertex set $V$ with $|V|=n\geq n_0$ and $n\in k\hn$. Let $\hht=\mathcal{T}(\hh_1,\ldots,\hh_{n/k})$. If $\hht$ is not $\epsilon$-close to $\hht_{ext}$, and $\delta_\ell(\hh_i)>\delta(n,k,\ell)$   for every  $1\leq i\leq n/k$. Then there exists a matching $\ha$ in $\hht$ with $|V(\ha)|\leq \gamma n $ such that for every balanced set $U\subset V(\hht)\setminus V(\ha)$ with $|U|\leq \gamma^{8}n$, there exists a matching $\hq$ in $\hht$ such that $V(\hq)=V(\ha)\cup U$.
\end{lem}

  We are now ready to prove Theorem \ref{thm-nonextr}.
\begin{proof}[Proof of Theorem \ref{thm-nonextr}]
  Using Lemma \ref{lem-absorbing}, we can find a matching $\ha$ in $\hht$ with $|V(\ha)|\leq \gamma n$. Removing this matching from $\hht$, we obtain a $(1,k)$-graph $\hht'$ on the vertex set $X'\cup V'$, where $\hht'=\hht-V(\ha)$, $V'=V\setminus V(\ha)$, and $X'=X\setminus V(\ha)$. For $i\in X'$,   let $\hh_i'=\hh_i-V(\ha)$. Let $n'=|V'|\geq n-\gamma n>n/2$. Recall that $\delta(n,k,\ell)=(1/2-o(1))\binom{n-\ell}{k-\ell}$,  we have
\begin{align*}
  \delta_\ell(\hh_i')
  &\geq \delta(n,k,\ell)-|V(\ha)|\binom{n'-\ell-1}{k-\ell-1}\\
  &>\left(\frac{1}{2}-o(1)-\gamma k\right)\binom{n'-\ell}{k-\ell}\\
  &>\left(\frac{k-\ell}{k}-\frac{1}{k^{k-\ell}}+\xi\right)\binom{n'-\ell}{k-\ell},
\end{align*}
when $n$ is sufficiently large  and  $\gamma k<\xi=\frac{1}{2k^{k-\ell}}$.

 Then by Lemma \ref{lem-pathcover}, there exists a matching $\hm$ in $\hht'$ cover all but at most $\sqrt{n}$ vertices. Let $U=V(\hht')\setminus V(\hm)$. Since $\sqrt{n}\ll \gamma^8 n$, by Lemma \ref{lem-absorbing}, there exists a matching $\mathcal{Q}$ such that $V(\mathcal{Q})=V(\ha)\cup U$. Thus $\hm\cup \mathcal{Q}$ forms a perfect matching of $\hht$.
\end{proof}

\subsection{Proof of Almost Cover Lemma}\label{sec:lem-pathcover}
The Almost Cover Lemma is a rainbow version of Lemma 2 in  \cite{Marks2011}, and the proof is very similar. For the sake of completeness, we include the full proof as well.

For two hypergraphs $\hh$ and $\hh'$, let $N(\hh,\hh')$ be the number of copies of $\hh'$ in $\hh$. Same as in \cite{Marks2011},  the following results are needed in the proof. 
\begin{lem}[\hspace{1sp}\cite{erdos1964}]\label{lem-2.4}
  For every integer $r\geq 2$, every $d>0$, and every $r$-partite $r$-graph $\hh'$, there exist $c>0$ and $n_0$ such that for every $r$-graph $\hh$ on $n\geq n_0$ vertices and $|\hh|\geq dn^r$, we have $N(\hh,\hh')\geq cn^{|V(\hh')|}$. 
\end{lem}
\begin{fact}[\hspace{1sp}\cite{Marks2011}]\label{fact-2.1}
  For all integer $k\geq 1$, $n\geq 2$, and $1\leq t\leq n-1$, the maximum number of edges in a $k$-partite $k$-graph with $n$ vertices in each class and no matching of $t+1$ is $tn^{k-1}$.
\end{fact}
\begin{lem}[\hspace{1sp}\cite{Pohoata2022}]\label{lem-Pohoata2022}
  Let $M_1,\ldots, M_N$ be matchings each of size $t$ in a $r$-partite $r$-uniform hypergraph. If $N>(t-1)t^r$, then there exists $i_1,\ldots, i_t$ and pairwise disjoint $e_{i_1}\in M_{i_1},\ldots, e_{i_t}\in M_{i_t}$.  
\end{lem}

\begin{proof}[Proof of  Lemma \ref{lem-pathcover}]
   Let $\hm$ be a matching in $\hht$ that maximizes the size $|\hm|$. Assume to the contrary that $n-|V(\hm)|\geq \sqrt{n}$, and let $X_1=X\setminus V(\hm)$ and $V_1=V\setminus V(\hm)$. 
   
   The proof strategy involves finding a matching $\hq$ in $\hht$ such that $\hq$ intersects at most $|\hq|-1$ elements in $\hm$. By replacing these elements with $\hq$, we increase the size of $\hm$.  It is precisely due to the use of this strategy that we can safely assume $n-|V(\hm)|=\sqrt{n}$; otherwise, we can augment $\hm$ by adding any $(k+1)$-balanced sets from $X_1\times V_1$ until $n-|V(\hm)|=\sqrt{n}$ is achieved.
   
   For each $S\in X_1\times \binom{V_1}{\ell}$, define
\[L_S(\hm):=\left\{T\in\binom{V(\hm)}{k-\ell}\colon S\cup T\in \hht, \forall e \in \hm, |T\cap e|\leq 1\right\} .\]
Note that  at most $o(n^{k-\ell})$ edges in $\mathcal{T}$ that contain $S$ and intersect any  edge in $\hm$ with more than one vertex. Moreover, at most $\sqrt{n}\binom{n-\ell-1}{k-\ell-1}$ edges intersecting at least one vertex outside $V(\hm)$. Due to the minimum $\ell$-degree condition, we have a lower bound 
\begin{align}\label{eq-2.1}
  |L_S(\hm)|>\left(\frac{k-\ell}{k}-\frac{1}{k^{k-\ell}}+\xi-o(1)\right)\binom{n-\ell}{k-\ell}.
\end{align}
  Clearly, each element in $L_S(\hm)$ intersect exactly $k-\ell$ edges in $\hm$. Then 
   \begin{align}\label{eq-2.2}
     L_S(\hm)=\cup_{\mathcal{E}\in \binom{\hm}{k-\ell}}L_S(\mathcal{E}).
   \end{align}
  We break the family $\binom{\hm}{k-\ell}$  into two parts 
   $
   \binom{\hm}{k-\ell}=A(S)\cup B(S),
   $
   where $A(S)=\{\mathcal{E}\in \binom{\hm}{k-\ell}\colon |L_S(\mathcal{E})|\leq (k-\ell)k^{k-\ell-1}-1\}$ and  $B(S)= \binom{\hm}{k-\ell}\setminus A(S)$.

The equation \eqref{eq-2.2} and the trivial bounds $|L_S(\mathcal{E})|\leq k^{k-\ell}$, $|A(S)|\leq \binom{|\hm|}{k-\ell}$, $|\hm|\leq n/k$   imply that 
\begin{align*}
  |L_S(\hm)|
  &\leq k^{k-\ell}|B(S)|+((k-\ell)k^{k-\ell-1}-1)\binom{|\hm|}{k-\ell}\\
  &\leq \left( \frac{|B(S)|}{\binom{|\hm|}{k-\ell}}+\frac{k-\ell}{k}-\frac{1}{k^{k-\ell}}\right)\binom{n-\ell}{k-\ell},
\end{align*}
Together with the lower bound \eqref{eq-2.1}, then
\begin{align}\label{eq-2.3}
 |B(S)|>(1-o(1))\binom{|\hm|}{k-\ell}.
\end{align}

According to Fact \ref{fact-2.1}, for any $\mathcal{E}\in B(S)$, the maximum matching size of $L_S(\mathcal{E})$ is at least $k-\ell$, and $k-\ell$ achieved only when $L_S(\mathcal{E})$ is isomorphic to a $(k-\ell)$-partite $(k-\ell)$-graph on these $k(k-\ell)$ vertices that consists of $\ell$ isolated vertices belonging to the same vertex part, along with all the remaining possible edges. Let us denote the set of $\mathcal{E}$ such that maximum matching size of $L_S(\mathcal{E})$ is  $k-\ell$   as $B_1(S)$, and $B_2(S)=B(S)\setminus B_1(S)$.


Starting from the minimum degree condition, we derived \eqref{eq-2.3}. Now, we will establish the following two claims based on the maximality of $|\hm|$.

\begin{claim}\label{claim--2.5}
For at most $\xi |X_1|\binom{|V_1|}{\ell}$ sets $S\in X_1\times \binom{V_1}{\ell}$, we have $|B_2(S)|\geq \frac{\xi}{3}\binom{|\hm|}{k-\ell}$.
\end{claim}
\begin{proof}[Proof of Claim]
Let us assume, for the sake of contradiction, that there are at least $\xi |X_1|\binom{|V_1|}{\ell}$ sets   $S\in X_1\times \binom{V_1}{\ell}$ such that $|B_2(S)|\geq \frac{\xi}{3}\binom{|\hm|}{k-\ell}$. 
Then, by averaging, there exists $\mathcal{E}_0\in \binom{\hm}{k-\ell}$ such that $\mathcal{E}_0\in B_2(S)$ for at least $\frac{\xi^2}{3}|X_1|\binom{|V_1|}{\ell}$ sets $S\in X_1\times \binom{V_1}{\ell}$. From these sets, recall that $|V_1|=\sqrt{n}$ and $|X_1|=\sqrt{n}/k$  we can choose a collection of pairwise disjoint sets $S_1, S_2, \ldots, S_N$ where $N>(k-\ell+1)^{r+1}$. According to Lemma \ref{lem-Pohoata2022}, by rearranging the indices there exist pairwise disjoint elements $T_1\in L_{S_1}(\mathcal{E}_0), T_2\in L_{S_2}(\mathcal{E}_0), \ldots, T_{k-\ell+1}\in L_{S_{k-\ell+1}}(\mathcal{E}_0)$.

Thus $\hm\setminus \mathcal{E}_0\cup \{S_1\cup T_1,\ldots,S_{k-\ell+1}\cup T_{k-\ell+1}\}$ forms a matching of $\hht$ of size larger than $\hm$.
This is a contradiction with the maximality of $|\hm|$.
\end{proof}
\begin{claim}\label{claim--2.6}
  For at most $\xi |X_1|\binom{|V_1|}{\ell}$ sets $S\in X_1\times \binom{V_1}{\ell}$, we have $|B_1(S)|\geq \frac{\xi}{3}\binom{|\hm|}{k-\ell}$.
\end{claim}
\begin{proof}[Proof of Claim]
Assume that there are at least $\xi |X_1|\binom{|V_1|}{\ell}$ sets   $S\in X_1\times \binom{V_1}{\ell}$ such that $|B_1(S)|\geq \frac{\xi}{3}\binom{|\hm|}{k-\ell}$. Then, there are at least $\frac{\xi^2}{6}\binom{|\hm|}{k-\ell}$ elements   $\mathcal{E}\in \binom{\hm}{k-\ell}$ such that each $\mathcal{E}\in B_1(S)$ for at least $\frac{\xi^2}{6}|X_1|\binom{|V_1|}{\ell}$ sets $S\in X_1\times \binom{V_1}{\ell}$.  From these $\frac{\xi^2}{6}\binom{|\hm|}{k-\ell}$ sets $\mathcal{E}\in \binom{\hm}{k-\ell}$, it is possible to choose three sets, $\mathcal{E}_1$, $\mathcal{E}_2$, and $\mathcal{E}_3$, that each pair of sets has precisely one common element, and the intersection between any two sets is distinct. Indeed, by Lemma \ref{lem-2.4} such configuration exist.

Without loss of generality, assume that $\mathcal{E}_1=\{f_{1},\ldots,f_{k-\ell-1},e_1\}$, $\mathcal{E}_2=\{e_1,\ldots,e_{k-\ell-1},g_1\}$ and $\mathcal{E}_3=\{g_{1},\ldots,g_{k-\ell-1},f_1\}$.   It follows from the definition that, for each $i=1,2,3$, there exist $\frac{\xi^2}{6}|X_1|\binom{|V_1|}{\ell}$ sets $S\in X_1\times \binom{V_1}{\ell}$ such that $\mathcal{E}_i\in B_1(S)$. Among these $S$, let us arbitrarily choose  $S^1, S^2, S^3$ satisfying $\mathcal{E}_i\in B_1(S^i)$. As mentioned previously, it is known that $L_{S^i}(\mathcal{E}_i)$ is isomorphic to a $(k-\ell)$-partite $(k-\ell)$-graph with $\ell$ isolated vertices. Importantly, these isolated vertices are all contained within the same partite set. Therefore, it is guaranteed that there exist two of $L_{S^i}(\mathcal{E}_i)$ where the isolated vertices are not in the same element of $\hm$. Without loss of generality, let us assume that the isolated vertices of $L_{S^1}(\mathcal{E}_1)$ and $L_{S^2}(\mathcal{E}_2)$ are not in the same element of $\hm$. Since $L_{S^1}(\mathcal{E}_1)$ and $L_{S^2}(\mathcal{E}_2)$ contain all the remaining possible edges, it follows that $L_{S^1}(\mathcal{E}_1)\cup L_{S^2}(\mathcal{E}_2)$ contains a matching $T_1^1,\ldots,T_{k-\ell}^1, T_1^2,\ldots,T_{k-\ell}^2$ such that $T_i^1\in L_{S^1}(\mathcal{E}_1)$ and $T_i^2\in L_{S^2}(\mathcal{E}_2)$ for $i=1,2\ldots,k-\ell$.

Recall that $\mathcal{E}_i\in B_1(S)$ for at least $\frac{\xi^2}{3}|X_1|\binom{|V_1|}{\ell}$ sets $S\in X_1\times \binom{V_1}{\ell}$. Moreover $|V_1|=\sqrt{n}$ and $|X_1|=\sqrt{n}/k$, from these sets we can choose a collection of pairwise disjoint sets $S_1^1, S_2^1, \ldots, S_N^1, S_1^2, S_2^2, \ldots, S_N^2$ where $\mathcal{E}_i\in B_1(S_j^i)$ for $i=1,2$ and $j=1,\ldots,N$,   $N>(k-\ell+1)^{r+1}$. According to Lemma \ref{lem-Pohoata2022}, by rearranging the indices there exist pairwise disjoint elements $T_j^i\in L_{S_{j}^i}(\mathcal{E}_i)$ for $j=1,\ldots, k-\ell$ and $i=1,2$.

Note that $|\mathcal{E}_1\cup \mathcal{E}_2|=2(k-\ell)-1$, thus   $(\hm\setminus (\mathcal{E}_1\cup \mathcal{E}_2))\cup \{S_j^i\cup T_j^i\}_{1\leq j\leq k-\ell, i=1,2}$ forms a matching in $\hht$ of size $|\hm|-|\mathcal{E}_1\cup \mathcal{E}_2|+2(k-\ell)=|\hm|+1$.
This is a contradiction with the maximality of $|\hm|$.
\end{proof}

The \eqref{eq-2.3} implies that for each $S\in X_1\times \binom{V_1}{\ell}$,  we have $|B_1(S)|+|B_2(S)|>(1-o(1))\binom{|\hm|}{k-\ell}$. Therefore either  half of $S\in X_1\times \binom{V_1}{\ell}$ satisfy $B_1(S)>\frac{1-o(1)}{2}\binom{|\hm|}{k-\ell}$ or half of $S\in X_1\times \binom{V_1}{\ell}$ satisfy $B_2(S)>\frac{1-o(1)}{2}\binom{|\hm|}{k-\ell}$, which contradict with at least one of Claim \ref{claim--2.5} and Claim \ref{claim--2.6}.  This concludes the proof of Lemma \ref{lem-pathcover}.
\end{proof}

\subsection{Proofs of the Absorbing Lemma.}\label{sec:lem-reser-absor}
%

Let $A$ be the vertex set of $1$ (or $2$) edges in $\hht$.  For a balanced $(k+1)$-set $E$, we say $A$ is an $1$- (or $2$-){\it absorber} for $E$ if there exists a matching $\mathcal{Q}$ in $\hht$ such that $V(\hq)=V(E)\cup V(A)$.

We prove Absorbing Lemma via the following counting lemma, which itself is proved in Section 2.3.
\begin{lem}[Counting]\label{lem-connecting}
  Given  $\epsilon>0$ and $k,\ell$ such that $k\geq 3$ and $k/2\leq \ell\leq k-1$,  there exists $\gamma$ and $n_0$ such that the following holds.  Suppose that $\hh_1,\ldots,\hh_{n/k}$ are $k$-graphs on vertex set $V$ with $|V|=n\geq n_0$ and $n\in k\hn$. Let $\hht=\mathcal{T}(\hh_1,\ldots,\hh_{n/k})$. If $\delta_\ell(\hh_i)>\delta(n,k,\ell)$   for every  $1\leq i\leq n/k$ and $\hht$ is not $\epsilon$-close to $\hht_{ext}$. Then for each balanced $(k+1)$-set $E$, there exists at least  $\gamma^5 n^{2(k+1)}$ $2$-absorbers.
\end{lem}

We also need a concentration inequality due to Frankl and Kupavskii.
\begin{lem}[Frankl-Kupavskii Concentration Inequality, \cite{frankl2018erd}]\label{thm-concentrate}
Suppose that $m,k,t$ are integers and $m\geq tk$. Let $\hg\subset \binom{[m]}{k}$ be a family, and $\theta=|\hg|/\binom{[m]}{k}$. Let $\eta$ be the random variable equal to the size of the intersection of $\hg$ with a $t$-matching $\hb$ of $k$-sets, chosen uniformly at random. Then $\ex[\eta]=\theta t$ and, for any positive $\gamma$, we have
\begin{align}\label{FK-ineq}
\Pr[|\eta-\theta t|\geq2\gamma\sqrt{t}]\leq 2e^{-\gamma^2/2}.
\end{align}
\end{lem}

\begin{proof}[Proof of Lemma \ref{lem-absorbing}]
 By Lemma \ref{lem-connecting}, for each balanced $(k+1)$-set $E$, there are at least $\gamma^5 n^{2(k+1)}$ $2$-absorbers of $E$. We denote the family of absorbers by $\ha(E)$.

Let $m=\gamma n/2(k+1) $. Let   $\hm'\subseteq \binom{V(\hht)}{2(k+1)}$ be a matching of size $m$ chosen uniformly at random. By Theorem \ref{thm-concentrate}, we have
\[
\Pr\left(\left||\hm'\cap \ha(E)|-\frac{\gamma^5 n^{2(k+1)}}{\binom{n+n/k}{2(k+1)}}m\right|>2\gamma^6m\right)<2e^{-\frac{\gamma^{12}}{2}m} , \quad \mbox{for all } E.
\]
Since $2e^{-\frac{\gamma^{12}}{2}m}<\frac{1}{ n^{2(k+1)}}$ for sufficiently large $n$, by the union bound, with probability more than 0, we can choose $\hm'$ such that for all $E$,
\begin{align}\label{eq-2.5new}
  |\hm' \cap \ha(E)|>\frac{\gamma^5 n^{2(k+1)}}{\binom{n+n/k}{2(k+1)}}m-2\gamma^6m>\gamma^7n.
\end{align}
Removing all non-absorbing $2(k+1)$-sets in $\hm'$, we get a matching $\hm$. 

In the following,  for any balanced set $U$ of size at most $(k+1)\gamma^8 n$, we are tring to  absorb $U$ by $\hm$. First,  part $U$ into balanced $(k+1)$-sets $E_1,E_2,\ldots,E_j$,   $1\leq  j\leq \gamma^8 n$.  By \eqref{eq-2.5new},  for each $i\in [j]$ there are at least $\gamma^7n$ $2$-absorbers for $E_i$ in $\hm$. Let us absorb $E_i$ by element in  $\hm$ step by step, in the $i$-th step, there are at least
\[
\gamma^7n-|U\cap A_1\cup\ldots\cup A_{i-1}|\binom{V(\hh)}{2k+1}>\gamma^7n-3(k+1)\gamma^8 n\binom{n+n/k}{2k+1}>0
\]
$2$-absorbers in $\hm$ disjoint to $U\cup A_1\cup\ldots\cup A_{i-1}$, choose one of it and denote by $A_i$. Thus, we obtain $\{A_1,A_2,\ldots,A_j\}\subset \hm-U$ and $A_i$ absorbs $E_i$. Note that $\hm$  is obtained by removing all non-absorbing $2(k+1)$-sets in $\hm'$, thus $\hm$ is consisted by $2$-absorbers.  Since each absorber is a matching in $\hht$ of size $2$, thus $\hm$ is a matching in $\hht$ with $|V(\hm)|\leq \gamma n$ and $\hm$ absorb any balanced set $U$ with $|U|\leq (k+1)\gamma^8 n$.
\end{proof}

\subsection{Proof of Counting Lemma}
We follow the similar approach as  \cite{Treglown2013}.

For any given balanced $(k+1)$-set $E$, and a partition $E=\{x\}\cup L\cup R$ such that $x\in X$, $|L|=\lceil k/2\rceil$, and $|R|=\lfloor k/2\rfloor$. Let $\ell=\binom{n}{\lceil k/2\rceil}$ and $r=\binom{n}{\lfloor k/2\rfloor}$.   Denote $\hf=\hhn_\hht(x)$.  By Proposition \ref{prop-2.1}, the degree assumption $\delta_\ell(\hf)>\delta(n,k,\ell)$  implies that
  \begin{align}\label{eq-2.5}
    \deg_{\hf}(L)>(\frac12-\frac12\gamma)r, \quad \deg_{\hf}(R)>(\frac12-\frac12\gamma)\ell,
  \end{align}
for any $\gamma>0$ and sufficiently large  $n$.

\begin{claim}[\hspace{1sp}\cite{Treglown2013}]\label{claim-2.10}
 Given $\epsilon>0$, there exists $\gamma>0$ and $n_0$ such that the following holds. Let $\hh$ be a $k$-graph on $V$, $|V|=n>n_0$.  If $\hh$ is not $\epsilon$-close to $\hh_{ext}$, then one of the following holds.
  \begin{itemize}
    \item [(a)] For any $L\in \binom{V}{\lceil {k}/{2}\rceil}$, there are at least $(\frac{1}{2}+\gamma)\ell$ $L'\in \binom{V}{\lceil k/2\rceil}$ such that $|\hhn_\hh(L)\cap \hhn_\hh(L')|\geq \gamma r$.
    \item [(b)] $|\{R'\in \binom{V}{\lfloor k/2 \rfloor}\colon |\hhn_\hh(R')|\geq (\frac12+\gamma)\ell\}|\geq 2\gamma r$.
  \end{itemize}
\end{claim}

\begin{claim}\label{claim-2.11}
  For any $x'\in X\setminus\{x\}$. Let $\hh=\hhn_\hht(x')$. If there are at least $\gamma\ell/2$ $L'\in  \hhn_{\hf}(R)$ such that $|\hhn_\hh(L)\cap \hhn_\hh(L')|\geq \gamma r$, then there are at least  $\frac{\gamma^3}{2} n^{k}$ $1$-absorbers $A$ for the set  $E$ with $|A\cap X|=\{x'\}$.
\end{claim}
\begin{proof}
  We first choose a $L'\in  \hhn_{\hf}(R)$ disjoint to $L\cup R$ such that  $|\hhn_\hh(L)\cap \hhn_\hh(L')|\geq \gamma r$,  the number of choices  is at least
  \[
  \gamma\ell/2-|L\cup R|\binom{n}{\lceil k/2\rceil-1}>\frac{\gamma}{3}\ell.
  \]
  Then we choose a $R'\in \hhn_\hh(L)\cap \hhn_\hh(L')$ disjoint to $L,R,L'$, the number of choices is at least
  \[
  |\hhn_\hh(L)\cap \hhn_\hh(L')|-|L\cup L'\cup R|\binom{n}{\lfloor k/2\rfloor-1}>\frac{\gamma}{2}r.
  \]
  Therefore $\{x'\}\cup L'\cup R'$ forms a $1$-absorbers of $E$, since $\{x'\}\cup L'\cup R'$,   $\{x\}\cup L'\cup R$, and $\{x'\}\cup L\cup R'$ are edges of $\hht$. The number of such absorbers is at least
  $
  \frac{\gamma}{3}\ell\frac{\gamma}{2}r>\gamma^3n^{k}.
  $
\end{proof}

\begin{claim}\label{claim-2.12}
  Given $x_1,x_2\in X\setminus \{x\}$, there are at least  $\gamma^3 n^{k}$ $1$-absorbers or at least $\gamma^4 n^{2k}$ $2$-absorbers $A$ satisfy $|A\cap X|\subseteq \{x_1,x_2\}$ for $E$.
\end{claim}
\begin{proof}
If case (a) in Claim \ref{claim-2.10} holds for  $\hh=\hhn_\hht(x_1)$ or $\hhn_\hht(x_1)$. Without lose of generality, assume holds for  $\hh=\hhn_\hht(x_1)$. Then there are at least $(\frac{1}{2}+\gamma)\ell$ $L'\in \binom{V}{\lceil k/2\rceil}$ such that $|\hhn_\hh(L)\cap \hhn_\hh(L')|\geq \gamma r$. Therefore, the number of $L'\in  \hhn_{\hf}(R)$ such that $|\hhn_\hh(L)\cap \hhn_\hh(L')|\geq \gamma r$ is at least 
\[
|\hhn_{\hf}(R)|+(\frac{1}{2}+\gamma)\ell-\ell>\frac{\gamma\ell}{2}.
\]
 By Claim \ref{claim-2.11}, there are at least $\gamma^3 n^{k}$ $1$-absorbers $ A$ for $E$  such that $| A\cap X|=\{x_1\}$. We are done.

Now we assume that case (b) in Claim \ref{claim-2.10} holds for  $\hh=\hhn_\hht(x_1)$. Let
\[
\hr:=\left\{R'\in \binom{V}{\lfloor k/2 \rfloor}\colon |\hhn_\hh(R')|\geq (\frac12+\gamma)\ell\right\}.
\]
Furthermore,  we can assume that
\begin{align}\label{eq-2.6new}
  |\hr\cap \hhn_\hh(L)|
  <\frac{\gamma r}{2}
\end{align}
and
\begin{align}\label{eq-2.7new}
  |\left\{L'\in \hhn_\hf(R)\colon |\hhn_\hh(L')\cap \hhn_\hh(L)|\geq \gamma r\right\}|
  <\frac{\gamma\ell}{2}
\end{align}
If \eqref{eq-2.7new} does not hold, then according to Claim \ref{claim-2.11},  there exist  at least $\frac{\gamma^3}{2} n^{k}$ $1$-absorbers $A$ for $E$ such that $|A\cap E|=\{x_1\}$. On the other hand, if \eqref{eq-2.6new} does not hold, then we have $|\left\{R'\in \hhn_\hf(L)\colon |\hhn_\hh(R')\cap \hhn_\hh(R)|\geq \gamma \ell\right\}|
>\gamma r/2$, which can be viewed as a mirror case of \eqref{eq-2.7new}.

First, choose $L'\in \hhn_\hf(R)$ disjoint to $L\cup R$ such that $|\hhn_\hh(L')\cap \hhn_\hh(L)|< \gamma r$. By \eqref{eq-2.5} and \eqref{eq-2.7new} the number of choices is at least
\[
(\frac{1}{2}-\frac{\gamma}{2})\ell-\frac{\gamma\ell}{2}>\frac{\gamma\ell}{3}.
\]
By \eqref{eq-2.5},\eqref{eq-2.6new}, and the fact that $|\hhn_\hh(L')\cap \hhn_\hh(L)|< \gamma r$, we can conclude that due to the definition $|\hr|>2\gamma r$, the following holds:
\begin{align}\label{eq-2.8new}
  |\hhn_\hh(L')\cap \hr|\geq \frac{\gamma r}{2}.
\end{align}
Second, choose $R'\in \hhn_\hh(L')\cap \hr $ disjoint to $L\cup R\cup L'$. By \eqref{eq-2.8new} the number of choices is at least $\frac{\gamma r}{3}$. Denote $\hg=\hhn_\hht(x_2)$. Third, choose  $R''\in \hhn_\hg(L)$ disjoint to $L\cup R\cup L'\cup R'$, the number of choices is at least $\frac{1}{3} r$. By \eqref{eq-2.5}, for each $R''\in \hhn_\hg(L)$,
\[
|\hhn_\hg(R'')\cap \hhn_\hh(R')|\geq (\frac12-\frac12 \gamma)\ell+ (\frac12+ \gamma)\ell-\ell=\frac{\gamma}{2}\ell.
\]
Last, choose $L''\in \hhn_\hg(R'')\cap \hhn_\hh(R') $ disjoint to $L\cup R\cup L'\cup R'\cup R''$, the number of choices is at least $\frac{\gamma}{3}\ell$. Thus, $\{x_1\}\cup L'\cup R', \{x_2\}\cup L''\cup R'' \in \hht$, and $\{x\}\cup L'\cup R, \{x_1\}\cup L''\cup R', \{x_2\}\cup L\cup R''\in \hht$. Therefore, $\{\{x_1\}\cup L'\cup R', \{x_2\}\cup L''\cup R''\} $ is a $2$-absorber of $E$. The choice number is at least
\[
\frac{\gamma}{3}\ell\frac\gamma 3  r\frac{1}{3} r\frac{\gamma}{3}\ell>\gamma^4n^{2k}.
\]
\end{proof}

There are $\binom{n/k-1}{2}$ choices for $\{x_1,x_2\}$, by Claim \ref{claim-2.12}  either for half of choices  $\{x_1,x_2\}\in \binom{X\setminus \{x\}}{2}$  there are at least  $\gamma^4 n^{k}$ $2$-absorbers such that $|A\cap X|= \{x_1,x_2\}$,  or half of choices of $\{x_1,x_2\}$  there are at least  $\gamma^3 n^{k}$ $1$-absorbers such that $|A\cap X|\in \{x_1,x_2\}$. The former implies that  there are at least
\[
\binom{n/k-1}{2} \gamma^4 n^{2k}>\gamma^5n^{2(k+1)}
\]
$2$-absorbers for $E$ in $\hht$.   The later implies that there are at least
\[
\binom{n/k-1}{2} \gamma^3 n^{k}/ n >\gamma^4n^{k+1}
\]
$1$-absorbers for $E$ in $\hht$, 
\begin{claim}\label{claim-2.13}
   If there are at least $m$ $1$-absorbers for $E$ in $\hht$, then there are at least $m\gamma n^{k+1}$ $2$-absorbers for $E$ in $\hht$.
\end{claim}
\begin{proof}
  By the degree assumption, for each $x\in X$, the size of $\hhn_\hht(x)$ is at least
  \[
  \deg_\hht(x)\geq \binom{n}{\ell}\left(\frac 12-\gamma\right)\binom{n-\ell}{k-\ell}/\binom{k}{\ell}>\frac{1}{3}\binom{n}{k}.
  \]
  Let $\ha$ be the family consists of all $1$-absorbers for $E$. For any $A\in \ha$, the number of edges  in $\hht$ that disjoint $A\cup E$ is  at least
  \[
  (|X|-2)\left(\deg_\hht(x)-|A\cup E|\binom{n-1}{k-1}\right)>\frac{1}{4}\binom{n+1}{k+1}.
  \]
   Each such  edge together with $A$ forms a $2$-absorber of $E$. Therefore, there are at least $m\frac{1}{4}\binom{n+1}{k+1}>m\gamma n^{k+1}$ $2$-absorbers for $E$.
\end{proof}
By Claim \ref{claim-2.13}, in either case there are at least $\gamma^5n^{2(k+1)}$ $2$-absorbers for $E$ in $\hht$.

\section{Extremal Case - Proof of Theorem \ref{thm-extr} }\label{sec:extr case}

Let $\hq$, $Q'$ be two $r$-graph on $V( \hq)$.    We say $x\in V( \hq)$ is  $\alpha$-good in $ \hq$ with respect to $ \hq'$ if $\deg_{ \hq'\setminus  \hq}(x)\leq \alpha \binom{|V( \hq)|-1}{r-1}$. Moreover, we say $ \hq$ is  $\alpha$-{\it good with respect to} $ \hq'$ if every vertex  is $\alpha$-good.

The following proposition shows that subgraphs inherit the property  ``good''. 
\begin{prop}[\hspace{1sp}\cite{Treglown2012}]\label{prop:3.2}
 Given reals $0<\alpha'<1$ and $0\leq c< 1$. Let  $\alpha:=\alpha'/c^{r-|S|}$. Suppose that  $S$ is $\alpha'$-good in $ \hq$ with respect to $ \hq'$. Let $ \hq''$ be a subgraph of $ \hq'$ on $U\subset V( \hq)$ such that $S\subset U$ and $|U|\geq cn$. Then $S$ is $\alpha$-good in $ \hq[U]$ with respect to $ \hq''$.
\end{prop}

\begin{prop}\label{claim:typical}
     Given real  $0< \epsilon\leq 1$ and  integer $1\leq j\leq r-1$. Let  $\epsilon':= \sqrt{r^r\epsilon} $. Suppose that  $ \hq$ is $\epsilon$-close to $ \hq'$. Then the number of not $\epsilon'$-good vertex is at most $\epsilon' |V( \hq)|$.
\end{prop}
\begin{proof}
  Let $m$ be the number of vertices that are not $\epsilon'$-good. Since $ \hq$ is $\epsilon$-close to $ \hq'$, there are at most $\epsilon |V( \hq)|^r$ edges in $ \hq'\setminus  \hq$. It follows that
  \[
  m\epsilon'\binom{|V( \hq)|-1}{r-1}\leq r| \hq'\setminus  \hq|\leq r\epsilon  |V( \hq)|^r.
  \]
  Then $m\leq r^r\epsilon |V( \hq)|/\epsilon'$. By setting $\epsilon'=\sqrt{r^r\epsilon}$, we conclude that $m\leq \epsilon'|V( \hq)|$.
\end{proof}

 By letting $\epsilon'=\sqrt{(k+1)^{k+1}\epsilon}$ and using Proposition \ref{claim:typical}, there are at most $\epsilon'(n+k/n)$ vertices in $X\cup V$ are not $\epsilon'$-good with respect to $\hht_{ext}$. Denote $U$ by the family consists of all vertices not $\epsilon'$-good. The following lemma allow us  to find a matching $\hm$ such that covering all the vertices not $\epsilon'$-good. Moreover, we will show later that $\hht-V(\hm)$ contains  a perfect matching.

\begin{lem}\label{lem:main parity}
	Given $k\geq 2$ and   $\epsilon>0$. There exists  $n_0$ such that the following holds. Suppose that $\hht$ is a $(1,k)$-graph on $X\cup V$, and $|V|=k|X|=n\geq n_0$. If $\hht$ is  $\epsilon$-close to  $\hht_{ext}$,  and $\delta_\ell(\hhn_\hht(x))>\delta(n,k,\ell)$ for each $x\in X$.  Then there exists a matching $\mathcal{M}$ in $\hht$ of size at most $|U|+1$ such that the following hold.  (i) Let $\hht'=\hht-V(\hm)$, $\hht'_{ext}=\hht-V(\hm)$, then $\hht'$ is $\epsilon'2^k$-good with respect to $\hht'_{ext}$, (ii) Recall that
   $
   \hht_{ext}= \mathcal{T}(\hh^i(A,B),\ldots,\hh^i(A,B))
   $
   for some $i\in \{0,1\}$ and $A,B$.  Let $A'=A\setminus V(\hm)$, $X'=X\setminus V(\hm)$, then
\begin{align}\label{eq-new3.1a}
  i|X'|\equiv |A'| \pmod 2.
\end{align}
\end{lem}
\begin{proof}
 We claim that there exists  $E\in \hht\setminus \hht_{ext}$ such that 
   \begin{align}\label{eq-3.2new}
     |E\cap A|\equiv i+1+|E\cap U| \pmod 2.
   \end{align}
Move the vertices in $U\cap A$ and $U\cap B$ to  the other part, we get a new partition $V=A_1\cup B_1$ where $A_1=(A\setminus U)\cup (B\cap U)$ and $B_1=(B\setminus U)\cup (A\cap U)$. Fix any $x\in X\setminus U$, by the degree assumption  $\delta_\ell(\hhn_\hht(x))> \delta(n,k,\ell)\geq\delta_\ell(\hh^i(A_1,B_1)) $,  there exists  $E_1\in  \binom{V}{k}$ such that  $E_1\in \hhn_\hht(x)\setminus \hh^i(A_1,B_1)$. Thus
\[
|E_1\cap A|-|E_1\cap A\cap U|+|E_1\cap B\cap U|=|E_1\cap A_1|\equiv i+1 \pmod 2.
\]
Therefore, $\{x\}\cup E_1=E$ satisfies \eqref{eq-3.2new}.

For $v_i\in U$, we greedily to find $E_i\in \hht$ such that $E_i\cap U=\{v_i\}$ and $E\setminus U,E_1,\ldots,E_i$ are pairwise disjoint. This is possible since in the $i$-th step there are at most $(k+1)|U|\leq \epsilon'(k+1)n$ vertices used in $(E\setminus U)\cup (U\setminus \{v_i\})\cup E_1\cup\ldots \cup E_{i-1}$. Then for $n$ sufficient large at least
  \[
  \epsilon' (n+n/k)^k-\epsilon'(k+1)n\binom{|X\cup V|-2}{k-1}\geq \frac{1}{2}\epsilon' n^k\geq 1
  \]
   elements in $\hhn_\hht(v_i)\setminus \hht_{ext}$ do not intersect  $(E\setminus U)\cup (U\setminus \{v_i\})\cup E_1\cup\ldots \cup E_{i-1}$. Let $\hm'=\{E_1,\ldots,E_{|U|}\}$.  Since  $E_i\in \hhn_\hht(v_i)\setminus \hht_{ext}$  for $i\in |U|$, then
   \begin{align}\label{eq-3.3new}
     |E_i\cap A|\equiv i+1\pmod 2.
   \end{align}
   
   If
   \begin{align}\label{eq-3.2a}
     (n/k-|U|)i\equiv |A\setminus V(\hm')| \pmod 2,
   \end{align} 
   then by  letting $\hm=\hm'$, the  (ii) follows. 
   
   Now assume \eqref{eq-3.2a} does not hold.  Let $\hm=\hm''\cup\{E\}$ where $\hm''=\hm'\setminus \{E_i\colon v_i\in E\cap U\}$. Therefore,
   by \eqref{eq-3.2new}, we have
   \[
   |A\setminus V(\hm)|\equiv|A\setminus V(\hm'')|-(i+1+|E\cap U|) \pmod 2.
   \]
   Then by \eqref{eq-3.3new}, it follows that
   \[
   |A\setminus V(\hm'')|\equiv |A\setminus V(\hm')|+|E\cap U|(i+1) \pmod 2.
   \]
   Since \eqref{eq-3.2a} not holds, we obtain that
   \begin{align*}
     |A\setminus V(\hm)|
     &\equiv |A\setminus V(\hm')|+|E\cap U|(i+1)-(i+1+|E\cap U|) \pmod 2\\
     &\equiv (n/k-|U|)i+1+(i+1+|E\cap U|)+|E\cap U|(i+1) \pmod 2\\
     &\equiv(n/k-|U|+|E\cap U|-1)i \pmod 2.
   \end{align*}
   Then, the (ii) follows. 
   
   Moreover, by Proposition \ref{prop:3.2} each $v$ is $\epsilon'2^k$-good with respect to $\hht'_{ext}$, the (i) follows.
\end{proof}

By using Lemma \ref{lem:main parity}, we are left with a ``good" sub-hypergraph which contains no ``bad" vertices.

In the following, we are going to divide $\hht'$ into two $(k+1)$-partite $(k+1)$-graph with large minimum vertex degree.  Recall that $\hht_{ext}=\mathcal{T}(\hh^i(A,B),\ldots,\hh^i(A,B))$ for some $A,B,i$ when $||A|-|B||\leq 1$ and $2\nmid in/k+|A|$. Let $\hm$ be a matching in $\hht$ described in \ref{lem:main parity}. Let $X'=X\setminus V(\hm)$, $A'=A\setminus V(\hm)$, $B'=B\setminus V(\hm)$, and $V'=A'\cup B'=V\setminus V(\hm)$, $n'=|V'|$.  Note that $V(\hm)\leq (|U|+1)(k+1)\leq 2\epsilon'kn$, then
\begin{align}\label{eq-3.1a}
  n/2-2\epsilon'kn\leq |A'|,|B'|\leq n/2+2\epsilon'kn.
\end{align}


We are looking for a partition of $V'=(S_1\cup\cdots\cup S_k)\cup (T_1\cup \cdots\cup T_k)\cup E$ and a partition of $X'=Y_1\cup Y_2\cup Y_3$ satisfy some conditions. According to the parity of $i$ and $k$, we distinguish four cases.
\begin{itemize}
  \item [(1)] If $i=0$ and $k$ is even.

      Let $|A'|= rk+s $ for some $0\leq s\leq k-1$. By Lemma \ref{lem:main parity},
      \begin{align}\label{eq-3.1}
        |A'|=rk+s\equiv in_1/k= 0  \pmod 2.
      \end{align}
      Since $k$ is even, we have $2\mid s$. Let $Y_3\cup E$ be an edge in $\hht'$ intersect $A'$ exactly $s$ elements. Such an edge exists since $\hht'$ is $\epsilon' 2^k$-good with respect to $\hht'_{ext}$.

      Partition $X'=Y_1\cup Y_2\cup Y_3$,   $A'\setminus E=T_1\cup\cdots\cup T_k$ and $B'\setminus E=S_1\cup\cdots\cup S_k$ such that $|Y_1|=|T_1|=\ldots=|T_k|=r$, and $|Y_2|=|S_1|=\ldots=|S_k|=\frac{n_1}{k}-r-1$.
  \item [(2)] If $i=0$ and $k$ is odd.

     Let $|A'|=r(k-1)+s$ for some $0\leq s\leq k-2$. By Lemma \ref{lem:main parity},
     \[
     |A'|=r(k-1)+s\equiv in_1/k=0 \pmod 2.
     \]
     Since $k$ is odd, we have $2\mid s$. Let $Y_3\cup E$ be an edge in $\hht'$ intersect $A'$ exactly $s$ elements. Such an edge exists since $\hht'$ is $\epsilon' 2^k$-good with respect to $\hht'_{ext}$.

      Partition $X'=Y_1\cup Y_2\cup Y_3$, $A'\setminus E=T_1\cup\cdots\cup T_{k-1}$ and $B'\setminus E=S_1\cup\cdots\cup S_k\cup T_k$ such that $|Y_1|=|T_1|=\ldots=|T_k|=r$, and $|Y_2|=|S_1|=\ldots=|S_k|=\frac{n_1}{k}-r-1$.

    \item [(3)] If $i=1$ and $k$ is even.

    Let $|A'|=n_1/k+r(k-2)+s$ for some $0\leq s\leq k-3$. By Lemma \ref{lem:main parity},
    \[
    |A'|=\frac{n_1}{k}+r(k-2)+s\equiv \frac{in_1}{k}=\frac{n_1}{k}  \pmod 2.
    \]
    Since $k$ is even, we have $s$ is even.

    Let   $Y_3\cup E$ be union of two  edges in $\hht'$ one intersect $A'$ exactly $1$ element and the other intersect $A'$ exactly $s-1$ elements. Such  edges exists since $\hht'$ is $\epsilon' 2^k$-good with respect to $\hht'_{ext}$.

      Partition $X'=Y_1\cup Y_2\cup Y_3$,  $A'\setminus E= T_1\cup\cdots\cup T_{k-1}\cup S_k$ and $B'\setminus E=S_1\cup\cdots\cup S_{k-1}\cup T_k$ such that $|Y_1|=|T_1|=\ldots=|T_k|=r$, and $|Y_2|=|S_1|=\ldots=|S_k|=\frac{n_1}{k}-r-2$.

    \item [(4)] If $i=1$ and $k$ is odd.

    We can seen $A'$ as $B'$ and $B'$ as $A'$. This case is same as case (2).

\end{itemize}

 In both four cases, the complete $(k+1)$-partite $(k+1)$-graph on $(Y_1,T_1,T_2,\ldots,T_k)$ and $(Y_2,S_1,S_2,\ldots,S_k)$  are  subgraphs of $\hht'_{ext}$.  Since \eqref{eq-3.1a} and $k\geq 3$,  those vertex  sets of size at least $n_1/3$. Then by Proposition \ref{prop:3.2}, the induced subgraph of $\hht'$ on $(Y_1,T_1,T_2,\ldots,T_k)$ and $(Y_2,S_1,S_2,\ldots,S_k)$  are $\epsilon'3^k$-good with respect to the complete $(k+1)$-partite $(k+1)$-graph on the same vertex set.

In fact, for a single vertex the property of being  ``good'' is equivalent to having a large degree.  The next lemma due to author joint with Wang allow us  to find a perfect matching in $r$-part $r$-graph with large degree.
\begin{lem}[\hspace{1sp}\cite{Wang2023}]\label{lem-k-part}
	For every integer $r\geq 2$, there exists $\alpha>0$ and $n_0$ such that the following holds. Suppose that $\hf$ is a $r$-partite $r$-graph with each part of size $n\geq n_0$ vertices, and $\delta(\hf)>(1-\alpha)n^{r-1}$.  Then  $\hf$ contains a perfect matching. \qed
\end{lem}

By using Lemma \ref{lem-k-part}, we get a perfect matching in the two $(k+1)$-partite $(k+1)$-graphs, which together with $Y_3\cup E$ forms a perfect matching in $\hht'$, further imply a perfect matching in $\hht$. The proof of Theorem \ref{thm-extr} is completed.

\end{document}